\documentclass[leqno,12pt]{article} 
\usepackage{amsmath, amssymb}
\usepackage{amsthm} 
\newcommand\supp{\mathop{\rm supp}}

\theoremstyle{plain} 
\newtheorem{theorem}{\indent\sc Theorem}[section]
\newtheorem{lemma}[theorem]{\indent\sc Lemma}
\newtheorem{corollary}[theorem]{\indent\sc Corollary}
\newtheorem{proposition}[theorem]{\indent\sc Proposition}

\theoremstyle{definition} 
\newtheorem{definition}[theorem]{\indent\sc Definition}
\newtheorem{remark}[theorem]{\indent\sc Remark}

\makeatletter
\def\address#1#2{\begingroup
\noindent\parbox[t]{7.8cm}{%
\small{\scshape\ignorespaces#1}\par\vskip1ex
\noindent\small{\itshape E-mail address}%
\/: #2\par\vskip4ex}\hfill%
\endgroup}%
\makeatother
\title{\uppercase{Weighted Calder\'on-Hardy spaces}} 
\author{
\textsc{Pablo Rocha} 
}
\date{} 
\begin{document}

\maketitle

\begin{abstract}
In this work we present the weighted Calder\'on-Hardy spaces on Euclidean spaces and investigate their properties. As an application we show, for certain power weights, that the iterated Laplace operator is a bijection from these spaces onto classical weighted Hardy spaces. The main tools to achieve our result are the atomic decomposition of weighted Hardy spaces, fundamental solutions of iterated laplacian and pointwise inequalities for certain maximal functions. 
\end{abstract}

\tableofcontents

\vspace{.2cm}

{\large References}


\section{Introduction}

It is well known that classical Hardy spaces $H^{p}(\mathbb{R}^{n})$ with $0 < p \leq 1$ are play an important role in the harmonic 
analysis and PDEs. Many important operators are better behaved on Hardy spaces $H^{p}(\mathbb{R}^{n})$  than on Lebesgue spaces 
$L^{p}(\mathbb{R}^{n})$ in the range $0 < p \leq 1$.

The Hardy spaces $H^{p}(\mathbb{R}^{n})$, $0 < p < \infty$, were first defined by Stein and Weiss \cite{Stein2} in terms of the theory of harmonic functions on $\mathbb{R}^{n}$. Afterward, Fefferman and Stein \cite{fefferman} introduced real variable methods into this subject and characterized the Hardy spaces by means of maximal functions. This second approach brought greater flexibility to the theory.

The spaces $H^{p}(\mathbb{R}^{n})$ can also be characterized by atomic decompositions. Roughly speaking, every distribution $f \in H^{p}$ can be expressed of the form
\begin{equation} \label{atomic series}
f = \sum_j \lambda_j a_j, 
\end{equation}
where the $a_j$'s are $p$ - atoms, $\{ \lambda_j \} \in \ell^{p}$ and $\| f \|^{p}_{H^{p}} \approx \sum_j | \lambda_j |^{p}$.
For $0 < p \leq 1$, an $p$ - atom is a function $a(\cdot)$ supported on a cube $Q$ such that
\[
\| a \|_{\infty} \leq |Q|^{-1/p} \,\,\  \text{and} \,\, \int x^{\alpha} a(x) dx = 0 \,\,\, \text{for all} \,\, |\alpha| \leq 
n \left( \frac{1}{p}-1 \right).
\]
Such decompositions were obtained by Coifman \cite{Coifman} for the case $n=1$ and by Latter \cite{Latter} for the case $n \geq 1$. These decompositions allow to study the behavior of certain operators on $H^{p}(\mathbb{R}^{n})$ by focusing one's attention on individual atoms. In principle, the continuity of an operator $T$ on $H^{p}(\mathbb{R}^{n})$ can often be proved by estimating $Ta$ when $a(\cdot)$ is an atom. For more results about Hardy spaces see \cite{Taibleson, Elias, Lu, Uchi, grafakos2}.

Gatto, Jim\'enez and Segovia \cite{segovia}, by using the atomic decomposition (\ref{atomic series}) for members of 
$H^{p}(\mathbb{R}^{n})$, solved the equation
\begin{equation} \label{eq}
\Delta^{m} F = f, \,\,\,\,\, \text{for} \,\, f \in H^{p}(\mathbb{R}^{n}).
\end{equation}
Moreover, they characterized the solution set of (\ref{eq}). These sets result be the Calder\'on-Hardy spaces 
$\mathcal{H}^{p}_{q, 2m}(\mathbb{R}^{n})$, which were defined by them for this purpose. More precisely, they proved that the iterated Laplace operator $\Delta^{m}$ is a {\bf bijective mapping} from Calder\'on-Hardy spaces $\mathcal{H}^{p}_{q, 2m}(\mathbb{R}^{n})$ onto Hardy spaces $H^{p}(\mathbb{R}^{n})$. They also  investigated the properties of these spaces and obtained an atomic decomposition for their elements.

The one-dimensional case with weights was studied by  Ombrosi \cite{sheldy}, there he introduced the one-sided Calder\'on-Hardy spaces 
$\mathcal{H}^{p,+}_{\alpha}((x_{-\infty}, +\infty), w)$ for weights $w$ in a Sawyer class and investigated their properties (see also
\cite{segovia2}). Perini \cite{perini} studied the boundedness of one-sided fractional integrals on these spaces. Ombrosi, Perini and Testoni 
\cite{sheldy1} obtained a complex interpolation theorem between one-sided Calder\'on-Hardy spaces. 

With the appearing of the theory of variable exponents the Hardy type spaces received a new impetus (see \cite{Orlicz, Kovacik, Diening2, Fiorenza, nakai, Cruz-Uribe2}). In this setting, the author \cite{rocha1} defined the variable Calder\'on-Hardy spaces 
$\mathcal{H}^{p(\cdot)}_{q, \gamma}(\mathbb{R}^{n})$, and studied the behavior of the iterated Laplace operator $\Delta^{m}$ on these spaces obtaining analogous results to those of Gatto, Jim\'enez and Segovia.

Recently, Auscher and Egert \cite{Egert} presented results on elliptic boundary value problems where the theory of Hardy spaces associated with operators plays a key role (see also \cite{Auscher} and references therein). 

The purpose of this paper is to define the weighted Calder\'on-Hardy spaces $\mathcal{H}^{p}_{q, \gamma}(\mathbb{R}^{n}, w)$ and investigate their properties. In Section 4 of this work we will prove our main results. These are contained in the following theorems.

\

{\sc Theorem} \ref{main result}. {\it Let $1 < q < \infty$, \, $n (2m + n/q)^{-1} <  p  \leq 1$ and let $0 < \mu  <  2m$  be such that 
$n  <  (2m+n/q - \mu) \, p$. If $w \in \mathcal{A}_{\frac{(2m+n/q - \mu)}{n}p}$, then the iterated Laplace operator $\Delta^{m}$ is a surjective mapping from $\mathcal{H}^{p}_{q, 2m}(\mathbb{R}^{n}, w)$ onto $H^{p}(\mathbb{R}^{n}, w)$. Moreover, there exist two positive constants $C_1$ and $C_2$ such that}
\begin{equation} \label{double ineq}
C_1 \|G \|_{\mathcal{H}^{p}_{q, 2m}(\mathbb{R}^{n}, w)} \leq \| \Delta^{m} G \|_{H^{p}(\mathbb{R}^{n}, w)} \leq C_2 
\|G \|_{\mathcal{H}^{p}_{q, 2m} (\mathbb{R}^{n}, w)}
\end{equation}
{\it hold for all} $G \in \mathcal{H}^{p}_{q, 2m}(\mathbb{R}^{n}, w)$.

\

The following theorem generalize to \cite[Theorem 1]{segovia}.

\

{\sc Theorem} \ref{main result 2}. {\it Let $w_a(x)= |x|^{a}$, $1 < q \leq r < \infty$, $n (2m + n/q)^{-1} <  p  \leq 1$, and let 
$0 < \mu  <  2m$ be such that $n  <  (2m+n/q - \mu) \, p$. If \,
$0 \leq a < \min \left\{ \frac{np}{r}, n \left(\frac{(2m+n/q - \mu)}{n}p - 1 \right) \right\}$, then the iterated Laplace operator 
$\Delta^{m}$ is a bijective mapping from $\mathcal{H}^{p}_{q, 2m}(\mathbb{R}^{n}, w_a)$ onto $H^{p}(\mathbb{R}^{n}, w_a)$, and 
(\ref{double ineq}) holds with $w=w_a$}.

\

The threshold $n (2m + n/q)^{-1}$ in the above results is optimal in the following sense:

\

{\sc Theorem} \ref{example}. {\it Let $w_a(x)= |x|^{a}$ with $-n < a < n(s-1)$ and $1 < s < \infty$. If 
$\displaystyle{p \leq \frac{n + \min\{ a, 0\}}{2m+n/q}}$ and $p \leq 1$, then $\mathcal{H}^{p}_{q, \, 2m}(\mathbb{R}^{n}, w_a) = \{ 0 \}.$}

\

This paper is organized as follows. In Section 2 we gives the basics of weighted Lebesgue theory, establish a Fefferman-Stein 
vector-valued inequality for the Hardy-Littlewood maximal operator and also recall the atomic decomposition of weighted Hardy spaces 
given in \cite{rocha2}. In Section 3 we define the weighted Calder\'on-Hardy spaces and investigate their properties.
The iterated Laplacian is also presented. Theorem \ref{main result} is proved in Section 4.

\

\textbf{Notation:} The symbol $A\lesssim B$ stands for the inequality $A \leq cB$ for some constant $c$. We denote by $Q( x_0, r)$ the cube centered at $x_0 \in \mathbb{R}^{n}$ with side lenght $r$. Given a cube $Q=Q(x_0, r)$ and $\delta >0$, we set $\delta Q = Q(x_0, \delta r)$. For a measurable subset $E\subseteq \mathbb{R}^{n}$ we denote by $\left\vert E\right\vert $ and $\chi_{E}$ the Lebesgue measure of $E$ and the characteristic function of $E$ respectively.  As usual we denote with $\mathcal{S}(\mathbb{R}^{n})$ the space of smooth and rapidly decreasing functions and with $\mathcal{S}'(\mathbb{R}^{n})$ the dual space (i.e.: the space of tempered distributions). A distribution $u$ acting on an element $\varphi \in \mathcal{S}(\mathbb{R}^{n})$ is denoted by $(u, \varphi)$. $\Delta$ and $\delta$ stand for the Laplacian and the Dirac's delta on $\mathbb{R}^{n}$ respectively. If $\mathbf{\alpha }$ is the multiindex $\alpha =(\alpha_{1},...,\alpha_{n})$ then 
$\left\vert \alpha \right\vert =\alpha _{1}+...+\alpha_{n}$. Given a real number $s \geq 0$, we write $\lfloor s \rfloor$ for the integer part of $s$.

Throughout this paper, $C$ will denote a positive constant, not necessarily the same at each occurrence.

\section{Preliminaries}

In this section we present weighted Lebesgue spaces and weighted Hardy spaces. For more information about these spaces the reader can consult \cite{Cruz-Uribe, Garcia2, grafakos} and \cite{Garcia, Torch} respectively. 

\subsection{Weighted Lebesgue spaces}

A weight $w$ is a non-negative locally integrable function on $\mathbb{R}^{n}$ that takes values in $(0, \infty)$ almost everywhere, i.e.: the weights are allowed to be zero or infinity only on a set of Lebesgue measure zero.

Given a weight $w$ and $0 < p < \infty$, we denote by $L^{p}(\mathbb{R}^{n}, w)$ the spaces of all functions $f$ defined on $\mathbb{R}^{n}$ satisfying 
$\| f \|_{L^{p}(\mathbb{R}^{n}, w)}^{p} := \int_{\mathbb{R}^{n}} |f(x)|^{p} w(x) dx < \infty$ . When $p=\infty$, we have that 
$L^{\infty}(\mathbb{R}^{n}, w) =L^{\infty}(\mathbb{R}^{n})$ with 
$\| f \|_{L^{\infty}(\mathbb{R}^{n}, w)} = \| f \|_{L^{\infty}(\mathbb{R}^{n})}$. If $E$ is a measurable set, we use the notation $w(E) = \int_{E} w(x) dx$. It is easy to check that $w(E)=0$ if and only if $|E|=0$.

It is well known that the harmonic analysis on weighted spaces is relevant if the weights $w$ belong to the class 
$\mathcal{A}_{p}$. Before defining the class $\mathcal{A}_p$, we first introduce the Hardy-Littlewood maximal operator.

Let $f$ be a locally integrable function on $\mathbb{R}^{n}$. The function
$$M(f)(x) = \sup_{Q \ni x} \frac{1}{|Q|} \int_{Q} |f(y)| dy,$$
where the supremum is taken over all cubes $Q$ containing $x$, is called the uncentered Hardy-Littlewood maximal function of $f$.

We say that a weight $w \in \mathcal{A}_1$ if there exists $C >  0$ such that
\begin{equation*}
M(w)(x) \leq C w(x), \,\,\,\,\, a.e. \, x \in \mathbb{R}^{n}, 
\end{equation*}
the best possible constant is denoted by $[w]_{\mathcal{A}_1}$. Equivalently, a weight $w \in \mathcal{A}_1$ if there exists $C >  0$ such that for every cube $Q$
\[
\frac{1}{|Q|} \int_{Q} w(x) dx \leq C \, ess\inf_{x \in Q} w(x). 
\]

For $1 < p < \infty$, we say that a weight $w \in \mathcal{A}_p$ if there exists $C> 0$ such that for every cube $Q$
$$\left( \frac{1}{|Q|} \int_{Q} w(x) dx \right) \left( \frac{1}{|Q|} \int_{Q} [w(x)]^{-\frac{1}{p-1}} dx \right)^{p-1} \leq C.$$
It is well known that $\mathcal{A}_{p_1} \subset \mathcal{A}_{p_2}$ for all $1 \leq p_1 < p_2 < \infty$. Also, if $w \in \mathcal{A}_{p}$ with $1 < p < \infty$, then there exists $1 < q <p$ such that $w \in \mathcal{A}_{q}$. This leads us to the following definition.

\begin{definition} Given $w \in \mathcal{A}_p$ with $1 < p < \infty$, we define the critical index of $w$ by
\[
q_{w} = \inf \{ q>1 : w \in \mathcal{A}_q \}.
\]
\end{definition}

\begin{remark} The index $q_w$ is related to the vanishing moment condition satisfied by the atoms for the atomic decompositions of the weighted Hardy spaces (see Definition \ref{atoms} below).
\end{remark}

The $\mathcal{A}_{p}-$weights, $1 < p < \infty$, give the following characterization for the Hardy-Littlewood maximal operator $M$:
\[
\int_{\mathbb{R}^{n}} [Mf(x)]^{p} w(x) dx \leq C \int_{\mathbb{R}^{n}} |f(x)|^{p} w(x) dx,
\]
for all $f \in L^{p}(\mathbb{R}^{n}, w)$ if and only if $w \in \mathcal{A}_p$ (see \cite[Theorem 9]{Muck}).

A weight $w$ satisfies the reverse H\"older inequality with exponent $s > 1$, denoted by $w \in RH_{s}$, if there exists $C> 0$ such that for every cube $Q$,
$$\left(\frac{1}{|Q|} \int_{Q} [w(x)]^{s} dx \right)^{\frac{1}{s}} \leq C \frac{1}{|Q|} \int_{Q} w(x) dx;$$
the best possible constant is denoted by $[w]_{RH_s}$. We observe that if $w \in RH_s$, then by H\"older's inequality, $w \in RH_t$ for all $1 < t < s$, and $[w]_{RH_t} \leq [w]_{RH_s}$. Moreover, if $w \in RH_{s}$, $s >1$, then $w \in RH_{s+ \epsilon}$ for some $\epsilon >0$. This gives the following definition. 

\begin{definition} Given $w \in RH_s$ with $s > 1$, we define the critical index of $w$ for the reverse H\"older condition by
\[
r_w = \sup \{ r >1 : w \in RH_{r} \}.
\]
\end{definition}

\begin{remark} The index $r_w$ is used to determine the size condition satisfied by the atoms for the atomic decompositions of the weighted Hardy spaces (see Definition \ref{atoms} below).
\end{remark}

In view of \cite[Corollary 7.3.4]{grafakos}, we define the class $\mathcal{A}_{\infty}$ by 
$\mathcal{A}_{\infty} = \bigcup_{1 \leq p < \infty} \mathcal{A}_p$. Since $w \in \mathcal{A}_{\infty}$ if and only if $w \in RH_{s}$ 
for some $s>1$, then it follows that $1 < r_w \leq +\infty$ for all $w \in \mathcal{A}_{\infty}$.

\vspace{.2cm}

The following lemma states a Fefferman-Stein vector-valued inequality for the Hardy-Littlewood maximal operator on 
$L^{p}(\mathbb{R}^{n}, w)$. This lemma is crucial to get Theorem 4.1.

\begin{lemma} \label{crucial lemma}
Let $1 < p < \infty$. Then for $u \in \left( 1,\infty \right) $ and $w \in \mathcal{A}_p$ we have that
\[
\left\Vert \left(\sum\limits_{j=1}^{\infty}\left(M f_{j}\right)^{u}\right)^{\frac{1}{u}} \right\Vert _{L^{p}(\mathbb{R}^{n}, w)} 
\lesssim \left\Vert \left( \sum\limits_{j=1}^{\infty }\left\vert f_{j}\right\vert^{u} \right)^{\frac{1}{u}}
\right\Vert _{L^{p}(\mathbb{R}^{n}, w)},
\]
holds for all sequences of bounded measurable functions with compact support $\left\{ f_{j}\right\}_{j=1}^{\infty}$.
\end{lemma}

\begin{proof} This lemma follows from \cite[Theorem 9]{Muck}, and \cite[Corollary 3.12]{Cruz-Uribe}.
\end{proof}

\subsection{Weighted Hardy spaces}

We gives the definition of weighted Hardy spaces and the atomic decomposition of these spaces developed by the author in 
\cite{rocha2}. We topologize $\mathcal{S}(\mathbb{R}^{n})$ by the collection of semi-norms $\{ p_{N} \}_{N \in \mathbb{N}}$ given by
$$p_{N}(\varphi) = \sum
\limits_{\left\vert \mathbf{\beta}\right\vert \leq N}\sup\limits_{x\in
\mathbb{R}^{n}}\left(  1+\left\vert x\right\vert \right)  ^{N}\left\vert
\partial^{\mathbf{\beta}}\varphi(x)\right\vert,$$
for each $N \in \mathbb{N}$. We set $\mathcal{F}_{N}=\left\{  \varphi\in \mathcal{S}(\mathbb{R}^{n}): p_{N}(\varphi) \leq 1 \right\}$. Let $f \in \mathcal{S}'(\mathbb{R}^{n})$, we denote by $\mathcal{M}_{\mathcal{F}_{N}}$ the grand maximal
operator given by
\[
\mathcal{M}_{\mathcal{F}_{N}}f(x)=\sup\limits_{t>0}\sup\limits_{\varphi\in\mathcal{F}_{N}
}\left\vert \left(  t^{-n}\varphi(t^{-1} \cdot)\ast f\right)  \left(  x\right)
\right\vert ,
\]
where $N$ is a large and fix integer.

\begin{definition}
Let $0 < p < \infty$. The weighted Hardy space $H^{p}(\mathbb{R}^{n}, w)$ is the set of all $f \in S'(\mathbb{R}^{n})$ for which 
$\mathcal{M}_{\mathcal{F}_{N}}f \in L^{p}(\mathbb{R}^{n}, w)$. In this case the "norm" of $f$ in $H^{p}(\mathbb{R}^{n}, w)$ is defined as 
\[
\left\Vert f\right\Vert _{H^{p}(\mathbb{R}^{n}, w)} := \left\Vert \mathcal{M}_{\mathcal{F}_{N}}f\right\Vert_{L^{p}(\mathbb{R}^{n}, w)}.
\]
\end{definition}

It is known that if $1 < p < \infty$ and $w \in \mathcal{A}_{p}$, then $H^{p}(\mathbb{R}^{n}, w) \simeq L^{p}(\mathbb{R}^{n}, w)$. In the range $0 < p \leq 1$ these spaces are not comparable.

Now, we introduce our atoms. We recall that our definition of atom differs from that given in \cite{Garcia, Torch}.

\begin{definition} \label{atoms} $(w-(p, p_{0}, d)$ atom$)$. Let $w \in \mathcal{A}_{\infty}$ with critical index $q_w$ and critical index $r_w$ for reverse H\"older condition. Let $0 < p \leq 1$, $\max \{ 1, p(\frac{r_w}{r_w -1}) \} < p_0 \leq \infty$,
and $d \in \mathbb{Z}$ such that $d \geq \lfloor n(\frac{q_w}{p} -1) \rfloor$, we say that a function $a(\cdot)$ is a $w - (p, p_0, d)$ atom centered in $x_0 \in \mathbb{R}^{n}$ if

\

$(a1)$ $a \in L^{p_0}(\mathbb{R}^{n})$ with support in the cube  $Q= Q(x_0, r)$.

\

$(a2)$ $\| a \|_{L^{p_0}(\mathbb{R}^{n})} \leq |Q|^{\frac{1}{p_0}} \left[ w(Q) \right]^{-\frac{1}{p}}$.

\

$(a3)$ $\int x^{\alpha} a(x) \, dx  =0$ for all multi-index $\alpha$ such that $| \alpha | \leq d$.
\end{definition}

\vspace{0.2cm}

Indeed, a $w-(p, p_{0}, d)$ atom $a(\cdot)$ belongs to $H^{p}(\mathbb{R}^{n}, w)$ $($see \cite[Lemma 2.8]{rocha2}$)$.

\begin{remark} \label{cond p_0}
We observe that the condition $\max \{ 1, p(\frac{r_w}{r_w -1}) \} < p_0 < \infty$ implies that 
$w \in RH_{(\frac{p_0}{p})'}$. If $r_w = \infty$, then $w \in RH_{t}$ for each $1 < t < \infty$. So, if $r_w =  \infty$ and since 
$\displaystyle{\lim_{t \rightarrow \infty}} \frac{t}{t-1} = 1$ we put $\frac{r_w}{r_w - 1} =1$. For example, if $w \equiv 1$, then 
$q_w =1$ and $r_w =  \infty$ and the definition of atom in this case coincide with the definition of atom in the classical Hardy spaces.
\end{remark}

The set
\[
\widehat{\mathcal{D}}_{0} = \left\{ \phi \in \mathcal{S}(\mathbb{R}^{n}) : \widehat{\phi} \in C_{c}^{\infty}(\mathbb{R}^{n}) \,\, \textit{and} 
\,\,\, 0 \notin \supp ( \widehat{\phi} \,)  \right\}
\]
is dense in $H^{p}(\mathbb{R}^{n}, w)$, $0 < p  < \infty$, for every  $w \in \mathcal{A}_{\infty}$ (see Theorem 1, p. 103, in \cite{Torch} and Proposition 7.1.5 (9), p. 503-504, in \cite{grafakos}).

Finally, the atomic decomposition for $H^{p}(\mathbb{R}^{n}, w)$, $0 < p \leq 1$, established in \cite{rocha2} it is as follows.

\begin{theorem} \label{atomic decomp} (\cite[Theorem 2.9]{rocha2}) Let $f \in \hat{\mathcal{D}}_{0}$, and $0 < p \leq 1$. If $w \in \mathcal{A}_{\infty}$, then there exist a sequence of $w - (p, p_0, d)$ atoms $\{ a_j \}$ and a sequence of scalars $\{ \lambda_j \}$ with $\sum_{j} |\lambda_j |^{p} \leq c \| f \|_{H^{p}_{w}}^{p}$ such that $f = \sum_{j} \lambda_j a_j$, where the convergence is in $L^{p_0}(\mathbb{R}^{n})$.
\end{theorem}

The novelty in this theorem is the convergence in $L^{p_0}(\mathbb{R}^{n})$-norm of the weighted atomic series.

\begin{remark} \label{w-atomic decomp}
Since $\widehat{\mathcal{D}}_{0}$ is dense in $H^{p}(\mathbb{R}^{n}, w)$, $0 < p \leq 1$, a routine argument allows us to ensure that every member of $H^{p}(\mathbb{R}^{n}, w)$ has a weighted atomic decomposition as in Theorem \ref{atomic decomp}, where the convergence is in $\mathcal{S}'(\mathbb{R}^{n})$.

The atoms in \cite[Theorem 2.9]{rocha2} are supported on balls, this theorem still holds if we consider atoms supported on cubes instead of balls. 
\end{remark}

\section{Weighted Calder\'on-Hardy spaces}

\subsection{Basics of weighted Calder\'on-Hardy spaces $\mathcal{H}^{p}_{q, \, \gamma}(\mathbb{R}^{n}, w)$}

Let $L^{q}_{loc}(\mathbb{R}^{n})$, $1 < q < \infty$, be the space of all measurable functions $g$ on $\mathbb{R}^{n}$ that belong locally to $L^{q}$ for compact sets of $\mathbb{R}^{n}$. We endowed $L^{q}_{loc}(\mathbb{R}^{n})$ with the topology generated by the seminorms
$$|g|_{q, \, Q} = \left( |Q|^{-1} \int_{Q} \, |g(y)|^{q}\, dy \right)^{1/q},$$ where $Q$ is a cube in $\mathbb{R}^{n}$ and $|Q|$ denotes its Lebesgue measure.

For $g \in L^{q}_{loc}(\mathbb{R}^{n})$, we define a maximal function $\eta_{q, \, \gamma}(g; x)$ as
$$\eta_{q, \, \gamma}(g; x) = \sup_{r > 0} r^{-\gamma} |g|_{q, \, Q(x, r)},$$
where $\gamma$ is a positive real number and $Q(x, r)$ is the cube centered at $x$ with side length $r$. This type of maximal function was introduced by Calder\'on \cite{calderon}.

Let $k$ a non negative integer and $\mathcal{P}_{k}$ the subspace of $L^{q}_{loc}(\mathbb{R}^{n})$ formed by all the polynomials of degree at most $k$. We denote by $E^{q}_{k}$ the quotient space of $L^{q}_{loc}(\mathbb{R}^{n})$ by $\mathcal{P}_{k}$. If $G \in E^{q}_{k}$, we define the seminorm
$\| G \|_{q, \, Q} = \inf \left\{ |g|_{q, \, Q} : g \in G \right\}$. The family of all these seminorms induces on $E^{q}_{k}$ the quotient topology.

Given a positive real number $\gamma$, we can write $\gamma = k + t$, where $k$ is a non negative integer and $0 < t \leq 1$. This decomposition is unique.

For $G \in E^{q}_{k}$, we define the maximal function $N_{q, \, \gamma}(G; x)$ by
\[
N_{q, \, \gamma}(G; x) = \inf \left\{ \eta_{q, \, \gamma}(g; x) : g \in G \right\}.
\] 
The maximal function $N_{q; \, \gamma}(G; \, \cdot \,)$ associated with a class $G$ in $E_{k}^{q}$ is lower semicontinuous (see  \cite[Lemma 6]{calderon}).

\begin{lemma} \label{g representante} Let $G \in E^{q}_{k}$ with $N_{q, \, \gamma}(G; x_0) < \infty,$ for some $x_0 \in \mathbb{R}^{n}$. Then,

$(i)$ there exists a unique $g \in G$ such that $\eta_{q, \, \gamma} (g; x_0) < \infty$ and, therefore, $\eta_{q, \, \gamma} (g; x_0) = N_{q, \, \gamma}(G; x_0)$;

$(ii)$ for any cube $Q$, there is a constant $c$ depending on $x_0$ and $Q$ such that if $g$ is the unique representative of $G$ given in 
$(i)$, then
$$\|G\|_{q, \, Q} \leq |g|_{q, \, Q} \leq C \, \eta_{q, \, \gamma} (g; x_0) = C \, N_{q, \, \gamma}(G; x_0).$$

The constant $c$ can be chosen independently of $x_0$ provided that $x_0$ varies in a compact set.
\end{lemma}
\begin{proof} The proof is similar to that of Lemma 3 in \cite{segovia}.
\end{proof}

\begin{corollary} If $\{ G_{j} \}$ is a sequence of elements of $E^{q}_{k}$ converging to $G$ in 
$\mathcal{H}^{p}_{q, \, \gamma}(\mathbb{R}^{n}, w)$, then $\{ G_j \}$ converges to $G$ in $E^{q}_{k}$.
\end{corollary}

\begin{proof} For any cube $Q$, by $(ii)$ of Lemma 4, we have
\begin{eqnarray*}
\| G- G_j \|_{q, \, Q} & \leq & C \, [w(Q)]^{-1/p} 
\| \chi_{Q}(\cdot) \,\, N_{q, \, \gamma}(G - G_j; \, \cdot \, ) \|_{L^{p}(\mathbb{R}^{n}, w)} \\
& \leq & C \, \| G - G_j \|_{\mathcal{H}^{p}_{q, \, \gamma}(\mathbb{R}^{n}, w)},
\end{eqnarray*}
which proves the corollary.
\end{proof}

\begin{lemma} \label{serie N}
Let $\{ G_j \}$ be a sequence in $E^{q}_{k}$ such that for a given point $x_0$, the series $\sum_j N_{q, \, \gamma}(G_j; \, x_0 )$ is finite. Then,

$(i)$ the series $\sum_j G_j$ converges in $E_{k}^{q}$ to an element $G$ and 
\[
N_{q, \, \gamma}(G; \, x_0 ) \leq \sum_j N_{q, \, \gamma}(G_{j}; \, x_0 );
\]

$(ii)$ if $g_j$ is the unique representative of $G_j$ satisfying
$\eta_{q, \, \gamma} (g_j; x_0) = N_{q, \, \gamma}(G_j; x_0)$, then $\sum_j g_j$ converges in $L^{q}_{loc}(\mathbb{R}^{n})$ to a 
function $g$ that is the unique representative of $G$ satisfying $\eta_{q, \, \gamma} (g; x_0) = N_{q, \, \gamma}(G; x_0)$.
\end{lemma}

\begin{proof} The proof is similar to that of Lemma 4 in \cite{segovia}.
\end{proof}

\begin{proposition} \label{g distribucion}
If $g \in L^{q}_{loc}(\mathbb{R}^{n})$ and there is a point $x_0$ such that $\eta_{q, \, \gamma} (g ; x_0) < \infty$, then 
$g \in \mathcal{S}'(\mathbb{R}^{n})$.
\end{proposition}

\begin{proof}
We first assume that $x_0 = 0$. Given $\varphi \in \mathcal{S}(\mathbb{R}^{n})$ and $k > \gamma + n$ we have that 
$|\varphi(y)| \leq p_{k}(\varphi) (1+|y|)^{-k}$ for all $y \in \mathbb{R}^{n}$. So
\begin{eqnarray*}
\left| \int_{\mathbb{R}^{n}} g(y) \varphi(y) dy \right| & \leq & p_{k}(\varphi) \int_{Q(0,1)} |g(y)| (1+|y|)^{-k} dy \\\\
& + & p_{k}(\varphi) \sum_{j=0}^{\infty} \int_{Q(0, 2^{j+1}) \setminus Q(0, 2^{j})} |g(y)| (1+|y|)^{-k} dy \\\\
& \lesssim &  p_{k}(\varphi) \, \eta_{q, \gamma}(g; 0) \\\\
& + &  p_{k}(\varphi) \, \eta_{q, \gamma}(g; 0) \, \sum_{j=0}^{\infty} 2^{j(\gamma + n -k)}.
\end{eqnarray*}
Being $k > \gamma +n$ it follows that $g \in \mathcal{S}'(\mathbb{R}^{n})$. For the case $x_0 \neq 0$ we apply the translation operator 
$\tau_{x_0}$ defined by $(\tau_{x_0}g)(x) = g(x + x_0)$ and use the fact that $\eta_{q, \gamma}(\tau_{x_0}g; 0) = \eta_{q, \gamma}(g; x_0)$.
\end{proof}

Next, we define the weighted Calder\'on-Hardy spaces on $\mathbb{R}^{n}$. 

\begin{definition} Let $w \in \mathcal{A}_{\infty}$ and $0 < p \leq 1$, we say that an element $G \in E^{q}_{k}$ belongs to the weighted Calder\'on-Hardy space $\mathcal{H}^{p}_{q, \, \gamma}(\mathbb{R}^{n}, w)$ if the maximal function $N_{q, \, \gamma}(G; \, \cdot \, ) \in L^{p}(\mathbb{R}^{n}, w)$. The "norm" of $G$ in $\mathcal{H}^{p}_{q, \, \gamma}(\mathbb{R}^{n}, w)$ is defined as 
\[
\| G \|_{\mathcal{H}^{p}_{q, \, \gamma}(\mathbb{R}^{n}, w)} := \| N_{q, \, \gamma}(G; \, \cdot \, ) \|_{L^{p}(\mathbb{R}^{n}, w)}.
\]
\end{definition}

\begin{proposition} \label{completitud}
Let $w \in \mathcal{A}_{\infty}$. Then, the space $\mathcal{H}^{p}_{q, \, \gamma}(\mathbb{R}^{n}, w)$ is complete.
\end{proposition}

\begin{proof} It is enough to show that $\mathcal{H}^{p}_{q, \, \gamma}(\mathbb{R}^{n}, w)$ has the Riesz-Fisher property: given any sequence $\{ G_j \}$ in $\mathcal{H}^{p}_{q, \, \gamma}(\mathbb{R}^{n}, w)$ such that 
$$\sum_{j} \| G_j \|_{\mathcal{H}^{p}_{q, \, \gamma}(\mathbb{R}^{n}, w)}^{p} < \infty,$$
the series $\sum_{j} G_j$ converges in $\mathcal{H}^{p}_{q, \, \gamma}(\mathbb{R}^{n}, w)$. 
For $1 \leq l$ fix, we have
\begin{eqnarray*}
\left\| \sum_{j=l}^{k} N_{q, \, \gamma}(G_{j}; \, . ) \right\|_{L^{p}(\mathbb{R}^{n},w)}^{p} & \leq & \sum_{j=l}^{k} \left\| N_{q, \, \gamma}(G_{j}; \, . ) \right\|_{L^{p}(\mathbb{R}^{n},w)}^{p} \\
& \leq & \sum_{j=l}^{\infty} \| G_j \|_{\mathcal{H}^{p}_{q, \, \gamma}(\mathbb{R}^{n}, w)}^{p} =: \alpha_l < \infty,
\end{eqnarray*}
for all $k \geq l$, thus $$\int_{\mathbb{R}^{n}} \, \left( \alpha_{l}^{-1/p} \, \sum_{j=l}^{k} N_{q, \, \gamma}(G_{j}; \, x ) \right)^{p} 
w(x) \, dx$$  $$\leq \int_{\mathbb{R}^{n}} \left( \left\| \sum_{j=l}^{k} N_{q, \, \gamma}(G_{j}; \, . ) \right\|_{L^{p}(\mathbb{R}^{n},w)}^{-1} \, \sum_{j=l}^{k} N_{q, \, \gamma}(G_{j}; \, x ) \right)^{p} w(x) \, dx =1, \,\,\, \forall \, k \geq l$$ it follows from Fatou's lemma as $k \rightarrow \infty$ that
$$\int_{\mathbb{R}^{n}} \, \left( \alpha_{l}^{-1/p} \, \sum_{j=l}^{\infty} N_{q, \, \gamma}(G_{j}; \, x ) \right)^{p} w(x) \, dx \leq 1$$
thus
\begin{equation}
\left\|  \sum_{j=l}^{\infty} N_{q, \, \gamma}(G_{j}; \, \cdot) \right\|_{L^{p}(\mathbb{R}^{n},w)}^{p} \leq \alpha_{l} = \sum_{j=l}^{\infty} \| G_j \|_{\mathcal{H}^{p}_{q, \, \gamma}(\mathbb{R}^{n}, w)}^{p} < \infty, \,\,\,\, \forall \, l \geq 1  \label{serie}.
\end{equation}
Taking $l=1$ in (\ref{serie}) and since $w \in \mathcal{A}_{\infty}$, we obtain that $\sum_{j} N_{q, \, \gamma}(G_{j}; \, x)$ is finite a.e. $x \in \mathbb{R}^{n}$. Then, by $(i)$ of Lemma \ref{serie N}, the series $\sum_j G_j$ converges in $E_{k}^{q}$ to an element $G$. Now 
$$N_{q, \, \gamma}\left( G - \sum_{j=1}^{k} G_j; x \right) \leq \sum_{j=k+1}^{\infty} N_{q, \, \gamma} (G_j; x),$$
from this and (\ref{serie}), we get
$$\left\| G - \sum_{j=1}^{k} G_j \right\|_{\mathcal{H}^{p}_{q, \, \gamma}(\mathbb{R}^{n},w)}^{p} \leq \sum_{j=k+1}^{\infty} \| G_j \|_{\mathcal{H}^{p}_{q, \, \gamma}(\mathbb{R}^{n},w)}^{p},$$
since the right-hand side tends to $0$ as $k \rightarrow \infty$, the series $\sum_{j} G_j$ converges to $G$ in 
$\mathcal{H}^{p}_{q, \, \gamma}(\mathbb{R}^{n},w)$.
\end{proof}

The proof of the following result is an adaptation of the proof of Theorem 2  given in \cite{segovia}.

\begin{theorem} \label{example} Let $m \geq 1$ and $w_a(x)= |x|^{a}$ with $-n < a < n(s-1)$ and $1 < s < \infty$. If 
$\displaystyle{p \leq \frac{n + \min\{ a, 0\}}{2m+n/q}}$ and $p \leq 1$, then $\mathcal{H}^{p}_{q, \, 2m}(\mathbb{R}^{n}, w_a) = \{ 0 \}.$
\end{theorem}

\begin{proof} We observe that for $1 < s < \infty$ and $-n < a < n(s-1)$ the nonnegative function $w_a(x) = |x|^{a}$, 
$x \in \mathbb{R}^{n} \setminus \{ 0 \}$, is a weight in the Muckenhoupt class $\mathcal{A}_{s}$ (see p. 506 in \cite{grafakos}). Let $F \in \mathcal{H}^{p}_{q, \, 2m}(\mathbb{R}^{n}, w_a)$  and assume $F \neq 0$. Then there exists $g \in F$ that is not a polynomial of degree less or equal to $2m-1$. It is easy to check that there exist a positive constant $C$ and a cube $Q = Q(0,r)$ with $r>1$ such that
$$\int_{Q} |g(y) - P(y)|^{q} \, dy \geq C > 0,$$
for every $P \in \mathcal{P}_{2m-1}$.

Let $x$ be a point such that $|x| > \sqrt{n}r$ and let $\delta= 4 |x|$. Then $Q(0,r) \subset Q(x, \delta)$. If $f \in F$, then $f = g - P$ for some $P \in \mathcal{P}_{2m-1}$ and
$$\delta^{-2m}|f|_{q, Q(x, \delta)} \geq C |x|^{-2m-n/q}.$$
So $N_{q,2m}(F;x) \geq C \, |x|^{-2m-n/q}$, for $|x| > \sqrt{n}r$. Since $\displaystyle{p \leq \frac{n + \min\{ a, 0\}}{2m+n/q}}$ and 
$-n  < a < n(s-1)$, we have
$$\int_{\mathbb{R}^{n}} \left[ N_{q,2m}(F; x) \right]^{p} w_{a}(x) \, dx \geq C \, \int_{|x| > \sqrt{n}r} |x|^{-(2m+n/q)p} |x|^{a} \, dx = \infty.$$
Then, we get a contradiction. Thus $\mathcal{H}^{p}_{q, 2m}(\mathbb{R}^{n}, w_a) = \{0\}$, if 
$\displaystyle{p \leq \frac{n + \min\{ a, 0\}}{2m+n/q}}$.
\end{proof}

\subsection{The iterated Laplacian}

The Laplace operator or Laplacian $\Delta$ on $\mathbb{R}^{n}$ is defined by
\[
\Delta = \sum_{j=1}^{n} \frac{\partial^{2}}{\partial x_j^{2}}.
\]
For $m \in \mathbb{N}$ we define the iterated Laplacian by $\Delta^{m} := \Delta \circ \cdot \cdot \cdot \circ \Delta$ ($m$ times). 
Given $g \in \mathcal{S}'(\mathbb{R}^{n})$, the iterated laplacian $\Delta^{m}$ acts on the distribution $g$ by means of the formula
\[
\left( \Delta^{m} g, \varphi \right) = \left(  g,  \Delta^{m} \varphi \right), \,\,\, \forall \, \varphi \in \mathcal{S}(\mathbb{R}^{n}).
\]
Thus $\Delta^{m}g \in \mathcal{S}'(\mathbb{R}^{n})$ when $g \in \mathcal{S}'(\mathbb{R}^{n})$.

Let $\Phi$ be the function defined on $\mathbb{R}^{n} \setminus \{ 0 \}$ by
\begin{equation}
\Phi(x) = \left\{ \begin{array}{cc}
                 C_1 \, |x|^{2m-n} \ln{|x|}, & \text{if} \,\, n \,\, \text{is even and} \,\, 2m-n \geq 0  \\
                 C_2 \, |x|^{2m-n}, & \text{otherwise}
               \end{array} \right. , \label{func h}
\end{equation}
a such function is a fundamental solution of $\Delta^{m}$ (see \cite{Gelfand}, p. 201-202); i.e.: $\Delta^{m} \Phi = \delta$ in 
$\mathcal{S}'(\mathbb{R}^{n})$.

\begin{lemma} \label{h function}
$($Lemma 8 in \cite{segovia}$)$ If $\Phi$ is the kernel defined in $($\ref{func h}$)$ and $|\alpha|=2m$, then $(\partial^{\alpha} \Phi)(x)$ is a $C^{\infty}$ homogeneous function of degree $-n$ on $\mathbb{R}^{n} \setminus \{0\}$ such that
\[
\int_{|x|=1} \, (\partial^{\alpha} \Phi)(x) \, dx = 0.
\]
\end{lemma}

Given a bounded function $a(\cdot)$ with compact support, its potential $b$, defined as 
\[
b(x) = \int_{\mathbb{R}^{n}} \Phi(x-y) a(y) dy,
\]
is a locally bounded function and $\Delta^{m} b = a$ in the sense of distributions. For these potentials, we have the following two results.

\begin{lemma} \label{b function} Let $a(\cdot)$ be an $(p,p_{0},d)-$atom with $d =\max\left\{\lfloor n(\frac{q_w}{p} -1) \rfloor, 2m-1 \right\}$ and assume that $Q(x_0, r)$ is the cube containing the support of $a(\cdot)$ in the definition of $(p,p_{0},d)-$atom. If $$b(x) = \int_{\mathbb{R}^{n}} \Phi(x-y) a(y) dy,$$ then for $|x - x_0| \geq \sqrt{n}r$ and every multiindex $\alpha$, there exists $C_{\alpha}$ such that $$\left| (\partial^{\alpha}b)(x) \right| \leq C_{\alpha} \, r^{2m+n} [w(Q)]^{-1/p} |x - x_0|^{-n-\alpha}$$ holds.
\end{lemma}

\begin{proof} Since $a(\cdot)$ has vanishing moments up to the order $d \geq 2m-1$, we have
$$(\partial^{\alpha}b)(x) = \int_{Q(x_0, r)} \, (\partial^{\alpha} \Phi)(x-y) a(y) \, dy$$
$$= \int_{Q(x_0, r)} \, \left[(\partial^{\alpha} \Phi)(x-y) - \sum_{|\beta| \leq 2m-1} (\partial^{\alpha + \beta} \Phi)(x-x_0) \frac{(x_0 -y)^{\beta}}{\beta !} \right] a(y) \, dy$$
$$= \int_{Q(x_0, r)} \, \left[\sum_{|\beta| = 2m} (\partial^{\alpha + \beta} \Phi)(x- \xi) \frac{(x_0 -y)^{\beta}}{\beta !} \right] a(y) \, dy$$
where $\xi$ is a point between $y$ and $x_0$. If $|x - x_0| \geq \sqrt{n}r$ it follows that $|x - \xi| \geq \frac{|x - x_0|}{2}$ since 
$| x_0 - \xi| \leq \frac{\sqrt{n}}{2} r$. Taking into account that for $|\beta| = 2m$, $\partial^{\alpha + \beta} \Phi$ is a homogeneous function of degree $-n-\alpha$, we obtain
\begin{eqnarray*}
| (\partial^{\alpha}b)(x) | & \leq & C \, r^{|\beta|} \int_{Q(x_0, r)} |x- \xi|^{-n-\alpha} |a(y)| dy \\\\
& \leq & C \, r^{|\beta|} \| a \|_{p_0} |Q|^{1-1/p_0} |x - x_0|^{-n-\alpha} \\\\
& \leq & C \, r^{2m + n} [w(Q)]^{-1/p}  |x - x_0|^{-n-\alpha}.
\end{eqnarray*}
\end{proof}

The following pointwise inequality is crucial to obtain the theorem \ref{main result} below.

\begin{proposition} \label{pointwise ineq}
Let $a(\cdot)$ be a $w-(p, p_{0}, d)$ atom with $d =\max\left\{\lfloor n(\frac{q_w}{p} -1) \rfloor, 2m-1 \right\}$
and assume that $Q=Q(x_0, r)$ is the cube containing the support of $a(\cdot)$ in the definition of $w-(p, p_{0}, d)$ atom.
 
If $b(x) = \int_{\mathbb{R}^{n}} \Phi(x-y) \, a(y) \, dy$, then for all $x \in \mathbb{R}^{n}$, all $0< \mu < 2m$ and all $q > 1$
\begin{equation} \label{N estimate}
N_{q, 2m}(B; x)  \lesssim [w(Q)]^{-1/p} \left[M(\chi_{Q})(x) \right]^{\frac{2m + n/q - \mu}{n}} + \chi_{4\sqrt{n}Q}(x) M(a)(x)
\end{equation}
\[
+ \, \chi_{4\sqrt{n}Q}(x) [M(M^{q}(a))(x)]^{1/q} + \chi_{4\sqrt{n}Q}(x) \sum_{|\alpha|=2m} T^{*}_{\alpha}(a)(x),
\]
where $B$ is the class of $b$ in $E^{q}_{2m-1}$, $T^{*}_{\alpha}(a) (x) = \sup_{\epsilon >0} \left|\int_{|x-y|> \epsilon} \, (\partial^{\alpha} \Phi)(x-y) a(y) \, dy\right|$ and $M$ is the Hardy-Littlewood maximal operator.
\end{proposition}

\begin{proof} We point out that the argument used in the proof of \cite[Proposition 15]{rocha1} to obtain the pointwise inequality
\cite[(4)]{rocha1} works in this setting as well, but considering now the conditions $(a1)$, $(a2)$ and $(a3)$ in Definition \ref{atoms} of 
$w-(p, p_0, d)$ atom (these conditions are similar to those of the atoms in variable context, see \cite{nakai} p. 3669). Then, this observation and Lemma \ref{b function} allow us to get (\ref{N estimate}).
\end{proof}

\section{Main result}

We observe that if $G \in \mathcal{H}^{p}_{q, \, 2m}(\mathbb{R}^{n}, w)$, then $N_{q, \, 2m}(G; x_0) < \infty,$ for some $x_0 \in \mathbb{R}^{n}$. By Lemma \ref{g representante} there exists $g \in G$ such that $N_{q, \, 2m}(G; x_0) = \eta_{q, \, 2m}(g; x_0)$; from Proposition \ref{g distribucion} it follows that $g  \in \mathcal{S}'(\mathbb{R}^{n})$. So $\Delta^{m}g$ is well defined in sense of distributions. On the other hand, since any two representatives of $G$ differ in a polynomial of degree at most $2m-1$, we get that $\Delta^{m} g$ is independent of the representative $g \in G$ chosen. Therefore, for $G \in \mathcal{H}^{p}_{q, \, 2m}(\mathbb{R}^{n}, w)$, we 
define $\Delta^{m} G$ as the distribution $\Delta^{m}g$, where $g$ is any representative of $G$.

\subsection{Solution of the equation $\Delta^{m} F =f$ for $f \in H^{p}(\mathbb{R}^{n}, w)$}

In this section, we shall prove that the space $\mathcal{H}^{p}_{q, \, 2m}(\mathbb{R}^{n}, w)$ is the solution set of the equation
\[
\Delta^{m} F = f, \,\,\,\, \text{for} \,\, f \in H^{p}(\mathbb{R}^{n}, w).
\] 
This is contained in the following two results.

\begin{theorem}\label{main result} Let $1 < q < \infty$, \, $n (2m + n/q)^{-1} <  p  \leq 1$ and let $0 < \mu  <  2m$  be such that 
$n  <  (2m+n/q - \mu) \, p$. If $w \in \mathcal{A}_{\frac{(2m+n/q - \mu)}{n}p}$, then the iterated Laplace operator $\Delta^{m}$ is a surjective mapping from $\mathcal{H}^{p}_{q, 2m}(\mathbb{R}^{n}, w)$ onto $H^{p}(\mathbb{R}^{n}, w)$. Moreover, there exist two positive constants $C_1$ and $C_2$ such that
\begin{equation} \label{double ineq 2}
C_1 \|G \|_{\mathcal{H}^{p}_{q, 2m}(\mathbb{R}^{n}, w)} \leq \| \Delta^{m}G \|_{H^{p}(\mathbb{R}^{n}, w)} \leq C_2 
\|G \|_{\mathcal{H}^{p}_{q, 2m} (\mathbb{R}^{n}, w)}
\end{equation}
hold for all $G \in \mathcal{H}^{p}_{q, 2m}(\mathbb{R}^{n}, w)$.
\end{theorem}

\begin{proof} Let $G \in \mathcal{H}^{p}_{q, \, 2m}(\mathbb{R}^{n}, w)$. Since $N_{q, 2m}(G; x)$ is finite a.e. $x \in \mathbb{R}^{n}$, 
by Proposition \ref{g distribucion}, the unique representative $g$ of $G$ (which depends on $x$) satisfying $\eta_{q, 2m}(g; x)=N_{q, 2m}(G; x)$ is a function 
in $L^{q}_{loc}(\mathbb{R}^{n}) \cap \mathcal{S}'(\mathbb{R}^{n})$. Thus, if $\phi \in \mathcal{S}(\mathbb{R}^{n})$ and 
$\int \phi(x) \, dx \neq 0$, from \cite[Lemma 6]{segovia} we get
\[
M_{\phi}(\Delta^{m}G)(x) \leq C \, p_{2m+n}(\phi) N_{q, \, 2m}(G; x).
\]
Thus $\Delta^{m} G \in H^{p}(\mathbb{R}^{n}, w)$ and 
\begin{equation} \label{continuity}
\| \Delta^{m} G \|_{H^{p}(\mathbb{R}^{n}, w)} \leq C \, \| G \|_{\mathcal{H}^{p}_{q, \, 2m}(\mathbb{R}^{n}, w)}.
\end{equation}

Now we shall see that the operator $\Delta^{m}$ is onto. Given $w \in \mathcal{A}_{\frac{(2m+n/q - \mu)}{n}p}$ and 
$f \in H^{p}(\mathbb{R}^{n}, w)$, by Remark \ref{w-atomic decomp}, there exist a sequence of nonnegative numbers 
$\{ \lambda_j \}_{j=1}^{\infty}$ and a sequence of cubes $\{Q_j \}_{j=1}^{\infty}$ and $w-(p, p_0, d)$ atoms $a_j$ supported on $Q_j$, 
such that $f= \sum_{j=1}^{\infty} \lambda_j a_j$ and
\begin{equation} \label{atomic ineq}
\sum_{j=1}^{+\infty} \lambda_{j}^{p} \lesssim \|f \|_{H^{p}(\mathbb{R}^{n}, w)}^{p}.
\end{equation}
For each $j \in \mathbb{N}$ we put $b_j(x)= \int_{\mathbb{R}^{n}} \Phi(x-y) a_j(y) dy,$ from Proposition \ref{pointwise ineq} we have

$$N_{q, 2m}(B_j; x)\lesssim  [w(Q_j)]^{-1/p} \left[M(\chi_{Q_j})(x) \right]^{\frac{2m + n/q - \mu}{n}} + \chi_{4\sqrt{n}Q_j}(x) M(a_j)(x)$$

$$+ \chi_{4\sqrt{n}Q_j}(x) [M(M^{q}(a_j))(x)]^{1/q} + \chi_{4\sqrt{n}Q_j}(x) \sum_{|\alpha|=2m} T^{*}_{\alpha}(a_j) (x).$$
So
$$\sum_{j=1}^{\infty} \lambda_j N_{q, 2m}(B_j; x) \lesssim \sum_{j=1}^{\infty} \lambda_j \frac{\left[M(\chi_{Q_j})(x) \right]^{\frac{2m + n/q - \mu}{n}}}{[w(Q_j)]^{1/p}} + \sum_{j=1}^{\infty} \lambda_j \chi_{4\sqrt{n}Q_j}(x) M(a_j)(x) $$ 

$$+\sum_{j=1}^{\infty} \lambda_j \chi_{4 \sqrt{n} Q_j}(x) [M(M^{q}(a_j))(x)]^{1/q} + \sum_{j=1}^{\infty} \lambda_j \chi_{4 \sqrt{n} Q_j}(x)  \sum_{|\alpha|=2m} T^{*}_{\alpha}(a_j) (x)$$ 

$$= I + II + III + IV.$$

To study $I$, by hypothesis, we have that $0 < p \leq 1$ and $(2m + n/q - \mu) p > n$. Then
\begin{eqnarray*}
\|I\|_{L^{p}(\mathbb{R}^{n}, w)} & = & \left\| \sum_{j=1}^{\infty} \lambda_j [w(Q_j)]^{-1/p} \left[M(\chi_{Q_j})(\cdot) \right]^{\frac{2m + n/q - \mu}{n}} \right\|_{L^{p}(\mathbb{R}^{n}, w)} \\
& = & \left\| \left\{ \sum_{j=1}^{\infty} \lambda_j [w(Q_j)]^{-1/p} \left[M(\chi_{Q_j})(\cdot) \right]^{\frac{2m + n/q - \mu}{n}} \right\}^{\frac{n}{2m + n/q - \mu}} \right\|_{L^{\frac{2m + n/q - \mu}{n}p}(\mathbb{R}^{n}, w)}^{\frac{2m + n/q - \mu}{n}} \\
& \lesssim & \left\| \left\{ \sum_{j=1}^{\infty} \lambda_j [w(Q_j)]^{-1/p} \chi_{Q_j}(\cdot)  \right\}^{\frac{n}{2m + n/q - \mu}} \right\|_{L^{\frac{2m + n/q - \mu}{n}p}(\mathbb{R}^{n}, w)}^{\frac{2m + n/q - \mu}{n}} \\
& = & \left\| \sum_{j=1}^{\infty} \lambda_j [w(Q_j)]^{-1/p} \chi_{Q_j}(\cdot) \right\|_{L^{p}(\mathbb{R}^{n}, w)} \\
& \lesssim & \left( \sum_{j=1}^{+\infty} \lambda_{j}^{p} \right)^{1/p} \lesssim \|f \|_{H^{p}(\mathbb{R}^{n}, w)},
\end{eqnarray*}
the first inequality follows from that $w \in \mathcal{A}_{\frac{(2m+n/q - \mu)}{n}p}$ and Lemma \ref{crucial lemma}, the condition $0 < p \leq 1$ gives the second inequality, and (\ref{atomic ineq}) gives the last one.

Next, we estimate $\| II\|_{L^{}(\mathbb{R}^{n}, w)}^{p}$. Since $0 < p \leq 1$, we have 
\begin{eqnarray*}
\| II\|_{L^{p}(\mathbb{R}^{n}, w)}^{p} & = & \left\| \sum_{j=1}^{\infty} \lambda_j \chi_{4 \sqrt{n} Q_j}(\cdot) M(a_j) (\cdot) \right\|_{L^{p}(\mathbb{R}^{n}, w)}^{p} \\
& \lesssim & \sum_{j=1}^{\infty} \lambda_{j}^{p} \int_{\mathbb{R}^{n}} \chi_{4 \sqrt{n} Q_j}(x) [M(a_j)(x)]^{p} w(x) dx
\end{eqnarray*}
applying Holder's inequality with $\frac{p_0}{p}$, where $\max \{ 1, p(\frac{r_w}{r_w -1}) \} < p_0 < +\infty$, we get
\begin{equation} \label{WM ineq}
\lesssim \sum_{j=1}^{\infty} \lambda_{j}^{p} \left( \int_{4 \sqrt{n}Q_j} [w(x)]^{(\frac{p_0}{p})'} dx \right)^{1 - \frac{p}{p_0}} 
\left( \int_{\mathbb{R}^{n}} [M(a_j)(x)]^{p_0} dx \right)^{\frac{p}{p_0}}.
\end{equation}
By Remark \ref{cond p_0} we have that $w \in RH_{(\frac{p_0}{p})'}$. Then, Proposition 7.1.5 (9) in \cite{grafakos} (p. 503-504), the boundedness of the maximal operator $M$ on $L^{p_0}(\mathbb{R}^{n})$, the condition $(a2)$ of the atom $a_j(\cdot)$, and 
(\ref{atomic ineq}) allow us to obtain
\[
\| II \|_{L^{p}(\mathbb{R}^{n}, w)} \lesssim \left( \sum_{j=1}^{+\infty} \lambda_{j}^{p} \right)^{1/p} \lesssim 
\|f \|_{H^{p}(\mathbb{R}^{n}, w)}.
\]

To study $III$, we apply once again H\"older's inequality with $\frac{p_0}{p}$ and obtain
\begin{eqnarray*}
\| III \|_{L^{p}(\mathbb{R}^{n}, w)}^{p} & \lesssim &  
\sum_{j=1}^{\infty} \lambda_{j}^{p} \int_{\mathbb{R}^{n}} \chi_{4 \sqrt{n} Q_j}(x) [M(M^{q}(a_j))(x)]^{\frac{p}{q}} w(x) dx \\
& \lesssim & \sum_{j=1}^{\infty} \lambda_{j}^{p} \left( \int_{4 \sqrt{n}Q_j} [w(x)]^{(\frac{p_0}{p})'} dx \right)^{1 - \frac{p}{p_0}} 
\left( \int_{\mathbb{R}^{n}} [M(M^{q}(a_j))(x)]^{\frac{p_0}{q}} dx \right)^{\frac{p}{p_0}}.
\end{eqnarray*}
Since we also can take $p_0 > q$, we have that the maximal operator $M$ is bounded on $L^{\frac{p_0}{q}}(\mathbb{R}^{n})$, and  
\[
\| III \|_{L^{p}(\mathbb{R}^{n}, w)}^{p} \lesssim 
\sum_{j=1}^{\infty} \lambda_{j}^{p} \left( \int_{4 \sqrt{n}Q_j} [w(x)]^{(\frac{p_0}{p})'} dx \right)^{1 - \frac{p}{p_0}} 
\left( \int_{\mathbb{R}^{n}} [M(a_j)(x)]^{p_0} dx \right)^{\frac{p}{p_0}}.
\]
Applying the same reasoning as the one carried out after of (\ref{WM ineq}) on the right-side hand of this inequality, one can conclude that
\[
\| III \|_{L^{p}(\mathbb{R}^{n}, w)} \lesssim \|f \|_{H^{p}(\mathbb{R}^{n}, w)}.
\]

Now, we study $IV$. By Lemma \ref{h function} above and Theorem 4 in \cite{stein} p. 42, we have that the operator $T_{\alpha}^{*}$ is bounded on $L^{p_0}(\mathbb{R}^{n})$. Proceeding as in the estimate of $II$, we get
\[
\| IV \|_{L^{p}(\mathbb{R}^{n}, w)} \lesssim \|f \|_{H^{p}(\mathbb{R}^{n}, w)}.
\]

Thus, the weighted estimates of $I, II, III$ and $IV$ give
$$\left\| \sum_{j=1}^{\infty} \lambda_j N_{q, 2m}(B_j; \, \cdot \, ) \right\|_{L^{p}(\mathbb{R}^{n}, w)} \lesssim \|f \|_{H^{p}(\mathbb{R}^{n}, w)}.$$
Hence
\begin{equation}
\sum_{j=1}^{\infty} \lambda_j N_{q, 2m}(B_j; x) < \infty \,\,\,\,\,\, \text{a.e.} \,\, x \in \mathbb{R}^{n}, \label{Nq}
\end{equation}
and
\begin{equation}
\left\| \sum_{j=M+1}^{\infty} \lambda_j N_{q, 2m}(B_j; \, \cdot \, ) \right\|_{L^{p}(\mathbb{R}^{n}, w)} \longrightarrow 0, \,\,\,\, \text{as} \,\,  M \rightarrow \infty  \label{Nq2}.
\end{equation}
From (\ref{Nq}) and Lemma \ref{serie N} follow that there exists a function $G$ such that $\sum_{j=1}^{\infty} k_j B_j = G$ in 
$E^{q}_{2m-1}$ and
$$N_{q, 2m} \left( \left(G - \sum_{j=1}^{M} k_j B_j \right) ; \, x \right) \leq C \sum_{j=M+1}^{\infty} k_j N_{q, 2m}(B_j; x).$$
This pointwise estimate together with (\ref{Nq2}) imply
$$\left\| G - \sum_{j=1}^{M} k_j B_j \right\|_{\mathcal{H}^{p}_{q,2m}(\mathbb{R}^{n}, w)} \longrightarrow 0, \,\,\,\, \text{as} \,\,  M \rightarrow \infty.$$
By proposition \ref{completitud}, we have that $G \in \mathcal{H}^{p}_{q,2m}(\mathbb{R}^{n}, w)$ and $G = \sum_{j=1}^{\infty} \lambda_j B_j$ in $\mathcal{H}^{p}_{q,2m}(\mathbb{R}^{n}, w)$. Since $\Delta^{m}$ is a continuous operator from $\mathcal{H}^{p}_{q,2m}(\mathbb{R}^{n}, w)$ into $H^{p}(\mathbb{R}^{n}, w)$ (see (\ref{continuity}) above), we get
\[
\Delta^{m} G = \sum_j \lambda_j \Delta^{m} B_j = \sum_j \lambda_j a_j = f,
\]
in $H^{p}(\mathbb{R}^{n}, w)$. This shows that $\Delta^{m}$ is onto $H^{p}(\mathbb{R}^{n}, w)$. Moreover,
\begin{equation} \label{first ineq}
\| G \|_{\mathcal{H}^{p}_{q,2m}(\mathbb{R}^{n}, w)} \lesssim \left\| \sum_{j=1}^{\infty} k_j N_{q, 2m}(B_j; \, \cdot \, ) 
\right\|_{L^{p}(\mathbb{R}^{n}, w)}
\end{equation}
\begin{eqnarray*}
\hspace{.3cm} & \lesssim & \|f \|_{H^{p}(\mathbb{R}^{n}, w)} \\\\
& = & \| \Delta^{m} G \|_{H^{p}(\mathbb{R}^{n}, w)}.
\end{eqnarray*}
Finally, (\ref{continuity}) and (\ref{first ineq}) give (\ref{double ineq 2}). This completes the proof. 
\end{proof}

The following theorem generalize to \cite[Theorem 1]{segovia}.

\begin{theorem} \label{main result 2} Let $w_a(x)= |x|^{a}$, $1 < q \leq r < \infty$, $n (2m + n/q)^{-1} <  p  \leq 1$, and let 
$0 < \mu  <  2m$ be such that $n  <  (2m+n/q - \mu) \, p$. If \,
$0 \leq a < \min \left\{ \frac{np}{r}, n \left(\frac{(2m+n/q - \mu)}{n}p - 1 \right) \right\}$, then the iterated Laplace operator 
$\Delta^{m}$ is a bijective mapping from $\mathcal{H}^{p}_{q, 2m}(\mathbb{R}^{n}, w_a)$ onto $H^{p}(\mathbb{R}^{n}, w_a)$, and 
(\ref{double ineq 2}) holds with $w = w_a$.
\end{theorem}

\begin{proof}
To establish the injectivity of $\Delta^{m}$ in $\mathcal{H}^{p}_{q, 2m}(\mathbb{R}^{n}, w_a)$, we need to introduce the space 
$\mathcal{N}^{r,q}_{2m}$ (see p. 564 in \cite{calderon}). Let $1 < q \leq r < \infty$ and $m \geq 1$, we say 
that $g \in \mathcal{N}^{r,q}_{2m}$ if $g \in L^{q}_{loc}(\mathbb{R}^{n})$ and a.e. $x \in \mathbb{R}^{n}$ there exists a 
polynomial $p_x (\cdot) \in \mathcal{P}_{2m-1}$ such that, for some measurable set $\mathcal{O}$ with $|\mathcal{O}| < \infty$, 
$x \to \eta_{q, 2m}((g + p_x)(\cdot) ; x ) \in L^{r}(\mathbb{R}^{n} \setminus \mathcal{O})$. We observe that if $G \in E^{q}_{2m-1}$ 
and $N_{q,2m}(G; \, \cdot) \in L^{r}(\mathbb{R}^{n} \setminus \mathcal{O})$, where $1 < q \leq r < \infty$ and $|\mathcal{O}| < \infty$, then by (i) of Lemma \ref{g representante} we have that $G \subset \mathcal{N}^{r,q}_{2m}$. Since $\Delta^{m} G = 0$ means that
$\Delta^{m} g = 0$ for all $g \in G$, by \cite[Lemma 9]{calderon}, it suffices to show that for a such $G$ there exists a measurable 
set $\mathcal{O}$ such that $|\mathcal{O}| < \infty$ and $N_{q, 2m}(G; \, \cdot) \in L^{r}(\mathbb{R}^{n} \setminus \mathcal{O})$. Given
$G \in \mathcal{H}^{p}_{q, 2m}(\mathbb{R}^{n}, w_a)$, let $\mathcal{O}$ be defined by
\[
\mathcal{O} = \{ x \in \mathbb{R}^{n} : [w_a(x)]^{1/p} N_{q, 2m}(G; x) > 1 \}.
\]
Since $N_{q, 2m}(G; \, \cdot) \in L^{p}(\mathbb{R}^{n}, w_a)$ it follows that $|\mathcal{O}| < \infty$. Now, for $r \geq q$, 
$0 \leq a < \frac{np}{r}$, and since $p \leq 1 < q$, we obtain
\begin{eqnarray*}
\int_{\mathcal{O}^{c}} [N_{q, 2m}(G; x)]^{r} \, dx & = & \int_{\mathcal{O}^{c}} [N_{q, 2m}(G; x)]^{r} [w_a(x)]^{r/p}  [w_a(x)]^{-r/p} \, dx \\
& \leq & \int_{\mathcal{O}^{c}} [N_{q, 2m}(G; x)]^{p} w_a(x)  [w_a(x)]^{-r/p} \, dx \\
& \leq & \int_{\mathbb{R}^{n}} [N_{q, 2m}(G; x)]^{p} w_a(x) \, dx + \int_{|x| \leq 1} [w_a(x)]^{-r/p} \, dx \\
& < & \infty.
\end{eqnarray*}
Thus $G \subset \mathcal{N}^{r,q}_{2m}$ for all $G \in \mathcal{H}^{p}_{q, 2m}(\mathbb{R}^{n}, w_a)$. In particular, this gives the injectivity of $\Delta^{m}$ in $\mathcal{H}^{p}_{q, 2m}(\mathbb{R}^{n}, w_a)$.

To finish, we observe that $0 \leq a < \min \left\{ \frac{np}{r}, n \left(\frac{(2m+n/q - \mu)}{n}p - 1 \right) \right\}$ implies that 
$w_a \in \mathcal{A}_{\frac{(2m+n/q - \mu)}{n}p}$ (see p. 506 in \cite{grafakos}). So, by applying Theorem \ref{main result} with 
$w = w_a$, we obtain the surjectivity of $\Delta^{m}$ onto $H^{p}(\mathbb{R}^{n}, w_a)$ and also (\ref{double ineq 2}).
\end{proof}

Thus, Theorem \ref{main result 2} allows us to conclude that given $f \in H^{p}(\mathbb{R}^{n}, w_a)$, where
$0 \leq a < \min \left\{ \frac{np}{r}, n \left(\frac{(2m+n/q - \mu)}{n}p - 1 \right) \right\}$, the equation $\Delta^{m} F = f$
has an unique solution in $\mathcal{H}^{p}_{q, 2m}(\mathbb{R}^{n}, w_a)$, namely: $F := (\Delta^{m})^{-1}f$.

\begin{remark} Theorem \ref{main result 2} does not hold in general for $\mathcal{H}^{p}_{q,\gamma}(\mathbb{R}, w)$ when $\gamma$ is not a natural number. Ombrosi \cite{sheldy} gave an example which Theorem \ref{main result 2} is not true for $\mathcal{H}^{p}_{q,\gamma}(\mathbb{R}, w)$ with $w \equiv 1$ and the operator $\left(\frac{d}{dx}\right)^{\gamma}$, when $0 < \gamma < 1 $ and $(\gamma + 1/q)p>1$.
\end{remark}

\begin{remark} An open question if there exists a weight $w$ that not be a power weight such that $\Delta^{m}$ be injective in
$\mathcal{H}^{p}_{q, 2m}(\mathbb{R}^{n}, w)$. Another question that may be more interesting that arise of this work is the following: there are analogous results to those established in Theorems \ref{main result} and \ref{main result 2} if we replace the iterated Laplace operator $\Delta^{m}$ by other operator? To give a positive answer, our new operator should have fundamental solutions and one should consider
 Calder\'on-Hardy and Hardy spaces associated to this operator.
\end{remark}

\bigskip
\address{
Departamento de Matem\'atica \\
Universidad Nacional del Sur \\
8000 Bah\'{\i}a Blanca - Buenos Aires \\
Argentina}
{pablo.rocha@uns.edu.ar}

\end{document}